\documentclass[12pt]{amsart}
\usepackage[latin1]{inputenc}
\usepackage{amsfonts}
\usepackage[toc,page,title,titletoc,header]{appendix}
\usepackage{graphicx,psfrag,epsfig,multirow,caption,subcaption}
\usepackage{amssymb,amsmath,amscd,amsthm,amssymb,verbatim,setspace}
\numberwithin{equation}{section}
\usepackage{mathrsfs,wrapfig}
\usepackage{indentfirst}
\usepackage{extarrows}
\usepackage{float}

\usepackage[colorlinks=true]{hyperref}

\newtheorem{theorem}{Theorem}[section]
\newtheorem{definition}[theorem]{Definition}

\newtheorem{proposition}[theorem]{Proposition}
\newtheorem{lemma}[theorem]{Lemma}

\newtheorem*{remark*}{Remark}

\numberwithin{equation}{section}

\newcommand{\dv}{\text{div}}

\newcommand{\eps}{\varepsilon}

\newcommand{\dt}{\delta}
\newcommand{\al}{\alpha}

\newcommand{\bn}{\mathbf{n}}

\newcommand{\bM}{\mathbf{M}}

\newcommand{\cM}{\mathcal{M}}
\newcommand{\cH}{\mathcal{H}}

\newcommand{\cC}{\mathcal{C}}

\newcommand{\cK}{\mathcal{K}}
\newcommand{\cT}{\mathcal{T}}

\newcommand{\cQ}{\mathcal{Q}}

\newcommand{\cX}{\mathcal{X}}

\newcommand{\R}{\mathbb{R}}
\newcommand{\bS}{\mathbb{S}}

\DeclareMathOperator{\pr}{\partial}

\title[Generic LSF]{Generic regularity of Level Set Flows with spherical singularity}
\author{Ao Sun, Jinxin Xue}
\address{Department of Mathematics, Lehigh University, Chandler-Ullmann Hall, Bethlehem, PA 18015}
\email{aos223@lehigh.edu}
\address{New Cornerstone Science Laboratory, Department of Mathematics, Rm A115, Tsinghua University, Haidian District, Beijing, 100084}
\email{jxue@tsinghua.edu.cn}
\date{\today}

\begin{document}
	\maketitle
	\begin{abstract}The sphere is well-known as the only generic compact shrinker for mean curvature flow (MCF). In this paper, we characterize the generic dynamics of MCFs with a spherical singularity. In terms of the level set flow formulation of MCF, we establish that generically the arrival time function of level set flow with spherical singularity has at most $C^2$ regularity.
	\end{abstract}
	
	\section{Introduction}
	
	A family of hypersurfaces $\{\bM_t\}_{t\in[0,T)}$ in $\R^{n+1}$ is flowing by mean curvature if it satisfies the equation $\pr_t x=-\mathbf{H}(x)$,
	where $x$ denotes the position vector and $\mathbf{H}$ denotes the mean curvature vector pointing to the outer normal direction.
	
	In the context of numerical analysis, Osher and Sethian \cite{OS} introduced the level set flow (LSF) method to interpret the mean curvature flow (MCF) as the level sets of some functions. Specifically, if we consider a function $f$ satisfying the equation
	\[
	\pr_t f=|\nabla f|\dv\left(\frac{\nabla f}{|\nabla f|}\right),
	\]
	then the level sets $\{x:f(x,t)=0\}=:\bM_t$ are hypersurfaces flowing by MCF. Evans-Spruck \cite{ES} and Chen-Giga-Goto \cite{CGG} used the viscosity method to rigorously establish the existence and uniqueness of weak solutions to the level set flow equation. Such weak solutions are called viscosity solutions, and they are Lipschitz. Moreover, this approach enables a formulation of weak MCF. For further advancements in this direction, we refer interested readers to \cite{I1, I2, Wh1}.
	
	A special assumption imposed on MCF is mean convexity, meaning that the initial hypersurface has nonnegative mean curvature everywhere. This assumption leads to fruitful consequences for mean convex MCF, as highlighted in works such as \cite{HS1, HS2, Wh2, Wh3, Wa, HK}, among others. The corresponding LSF has monotonic advancing fronts. In this scenario, an arrival time function can be defined as $u(x) = f(x,t) - t$, where $u$ satisfies the equation
	
	\[
	-1=|\nabla u|\dv\left(\frac{\nabla u}{|\nabla u|}\right),
	\]
	and the level set $\{x : u(x) = t\}$ represents the time slice $\mathbf{M}_t$ of the MCF. It is worth noting that this equation is not defined at $x$ where $\nabla u(x)=0$, and \cite{ES} also defined the solution in the viscosity sense. Throughout this paper, we exclusively focus on this viscosity arrival time formulation of LSF. 
	
	Understanding the singularities that arise in MCF is a central problem. When a singularity develops at the spacetime point $(0,0) \in \mathbb{R}^{n+1} \times \mathbb{R}$, Huisken \cite{H2} introduced the rescaled mean curvature flow (RMCF) given by $\partial_t x = -\left(H - \frac{\langle x, \mathbf{n} \rangle}{2}\right)\mathbf{n}$ to perform a spacetime rescaling. A limit surface of RMCF is called a shrinker. It is a stationary solution of the RMCF equation, and satisfies the equation $H - \frac{\langle x, \mathbf{n} \rangle}{2} = 0$. Consequently, a shrinker serves as a model for the singularity. 
	
	Among all the shrinkers, the simplest one is the sphere $\mathbb{S}^n:=S^n(\sqrt{2n})$ with radius $\sqrt{2n}$ centered at the origin. In \cite{H1}, Huisken proved that any convex hypersurface in $\R^{n+1}$ for $n\geq 2$ must shrink to a single spherical singularity under MCF. Later, Huisken \cite{H2} proved that the sphere is the only closed singularity model for mean convex MCF. Another notable contribution came from Colding and Minicozzi in \cite{CM1}, where they showed that among all closed singularity models of MCF, the sphere is the only variationally stable one.
	
	In the context of LSF, an important problem concerns the regularity of its solutions, as the \emph{a priori} regularity of a viscosity solution is only Lipschitz. We refer the readers to \cite{CM2, CM3} for the general regularity results of the arrival time function. Building upon his work on convex MCF, Huisken \cite{H3} established that the LSF corresponding to a convex MCF possesses $C^2$ regularity. Sesum \cite{Se} later constructed examples demonstrating that the LSF with a single spherical singularity may not have $C^3$ regularity. On the other hand, the classical example of the shrinking sphere illustrates that the regularity of the LSF can be $C^\infty$. Furthermore, Strehlke \cite{St1} constructed examples indicating that the regularity of the LSF with a single spherical singularity can exhibit arbitrarily high smoothness.
	
	Given the aforementioned results, it becomes natural to ask about the typical regularity of the LSF with a single spherical singularity. Our main theorem demonstrates that low regularity is the generic behavior of the LSF. To establish this, we consider the space $\mathcal{G}$ consisting of closed orientable mean convex hypersurfaces embedded in $\mathbb{R}^{n+1}$, where $n \geq 2$, such that the mean curvature flow starting from these hypersurfaces only generates a single spherical singularity. We equip $\mathcal{G}$ with the $C^r$-topology for any $r > 2$, following the framework introduced by White in \cite{Wh0}. For a given closed embedded orientable $C^r$ hypersurface $\Sigma \subset \mathbb{R}^{n+1}$, we consider a $C^r$ unit vector field $\mathbf p(x)$ defined over $\Sigma$, such that $\mathbf p(x)$ and $T_x\Sigma$ span $T_x\R^{n+1}$. Let $C^r(\Sigma)$ denote the space of $C^r$ functions defined on $\Sigma$. We can then find a $\delta_0 > 0$ such that for all $\|f\|_{C^r} \leq \delta_0$, the sets $\{x + f(x)\mathbf p(x) \ |\ x \in \Sigma\}$ are still closed embedded $C^r$ hypersurfaces. These hypersurfaces form a local chart of the Banach manifold comprising all nearby closed embedded 
	$C^r$ hypersurfaces.

	\begin{theorem}\label{thm:main}
		There is an open and dense subset of $\mathcal G$, such that the LSF starting from any hypersurface in such a set is $C^2$ but not $C^3$.
	\end{theorem}
	
	The well-known elliptic regularity states that a solution to a linear elliptic equation in a domain with a smooth boundary must be smooth. However, in the case of the arrival time formulation of the LSF, which satisfies a \emph{degenerate} elliptic equation, the solution may not be smooth. Theorem \ref{thm:main} highlights that, in a generic sense, the arrival time function has the lowest possible regularity. This observation reveals some fundamental differences between linear and nonlinear equations.
	
	Our proof of Theorem \ref{thm:main} relies on a key discovery made by Strehlke \cite{St1}, which establishes a connection between the regularity of the LSF at a spherical singularity and the asymptotic dynamics of the RMCF. Theorem \ref{thm:main} is a consequence of the following theorem, which finds the generic behavior of MCFs with spherical singularities. Let $\mathcal H$ be the space of hypersurfaces in $\R^{n+1}$ where the MCFs starting from these hypersurfaces only generate a single spherical singularity endowed with the $C^r$ topology as above. It is clear that $\mathcal G\subset\mathcal H$. It is important to note that in the definition of $\cH$, we do not assume mean convexity.
	
	\begin{theorem}\label{thm:main2}
		There exists an open and dense subset $\mathcal{R}$ of $\mathcal{H}$ with the following significant property. For any MCF starting from a hypersurface in $\mathcal R$, the corresponding RMCF converges to the sphere $\bS^{n}:=S^n(\sqrt{2n})$ with convergence rate $e^{-t/n}$ as $t\to\infty$.
	\end{theorem}
	
	Theorem \ref{thm:main2} provides a dynamic perspective. In our previous work \cite{SX1}, we proved that a non-spherical compact singularity can be avoided through a generic perturbation of the initial condition (see also \cite{CCMS}). Theorem \ref{thm:main2} concerns the generic behavior of MCFs with spherical singularities. The convergence rate $ e^{-t/n}$ in the theorem can be understood as follows. The linearized RMCF around the sphere $\bS^{n}$ leads to the operator $L = \Delta_{\mathbb{S}^{n}} + 1$, which has first three eigenvalues $1,1/2,-1/n$. The eigenfunctions of the eigenvalues $1$ and $1/2$ correspond to infinitesimal dilation and translations respectively, which are not essential for the study of MCF singularities. Therefore, the first nontrivial eigenvalue is $-1/n$, which gives rise to the exponential decay rate of $e^{-t/n}$. 
	
	We would like to emphasize that Theorem \ref{thm:main2} holds even when $n=1$, but Theorem \ref{thm:main} does not hold when $n=1$. In fact, Kohn-Serfaty \cite{KS} proved that the arrival time of curves in the plane is at least $C^3$.

	\subsection*{Idea of proof}
	Strehlke \cite{St1} established that the local behavior of the LSF near a spherical singularity is related to the asymptotics of the RMCF. In particular, if the RMCF is dominated by the third eigenfunction of the operator $L_{\mathbb{S}^n}$, then the LSF cannot have $C^3$ regularity. This corresponds to the slowest possible convergence rate, and we refer to such singularities as \emph{slow} singularities. Therefore, Theorem \ref{thm:main2} consists of two key statements: a stability theorem (Theorem \ref{thm:stability}) and a denseness theorem (Theorem \ref{thm:denseness}). The stability theorem asserts that a slow singularity is robust under initial perturbations. The denseness theorem indicates that a singularity that is not slow can be perturbed into a slow one by an arbitrarily small initial perturbation. To establish the denseness statement, it is necessary to demonstrate that after a generic perturbation, the RMCF will evolve toward the space spanned by the third eigenfunction.
	
	Indeed, either situation involves two key steps. The first step addresses a local problem, where the goal is to demonstrate that for a generic convex hypersurface that is sufficiently close to the sphere, the arrival time function is $C^2$ but not $C^3$. In other words, this step establishes that slow singularities are generic when we perturb the flow near the singularity.
	
	The second step tackles a global problem, aiming to show that for a generic \emph{initial} hypersurface, the arrival time function is $C^2$ but not $C^3$. This is equivalent to showing that slow singularities are generic when we perturb the initial data. However, the second step often presents challenges, particularly in demonstrating that a generic initial perturbation can propagate to the correct local perturbation observed in the first step. This difficulty arises because we cannot solve a (geometric) heat equation backwardly. Notably, there is a similarity between the current work and our previous works \cite{SX1, SX2}, where we prove that a generic initial perturbation makes the flow leave a nonspherical compact shrinker or a conical noncompact shrinker $\Sigma$ along the leading eigenfunction of $L_\Sigma$. In particular, for the second step mentioned above, we developed a Li-Yau estimate in \cite{SX1} and a Feynman-Kac estimate in \cite{SX2} to show a generic initial perturbation can propagate to a correct local perturbation. 
	
	However, a fundamental distinction arises within the framework of the present paper when considering the eigenvalue $-1/n$, as it is not the leading eigenvalue as in previous works. Therefore, neither the Li-Yau estimate of \cite{SX1} nor the Feynman-Kac representation of \cite{SX2} applies here, since in both cases, the proof relies strongly on the analysis of positive solutions of the linearized heat equation $\partial_t v=L_{M_t}v$ along the RMCF $M_t$, whereas the eigenfunction associated to the eigenvalue $-1/n$ here does not have a fixed sign. In light of this, we use a result established by Herrero-Vel\'azquez \cite{HV}, which shows that the fundamental solution to the linearized variational equation has a dense image. As a consequence, we can choose an initial perturbation to have a nontrivial projection to the eigenspace with eigenvalue $-1/n$ after long time evolution.
	
	Nevertheless, it is inevitable to have nontrivial projections onto the eigenspace with eigenvalues $1$ and $1/2$ as well. The subsequent major challenge lies in addressing the first two eigenfunctions, namely the constant function and $\langle V,\bn\rangle$ for any constant vector $V$. These eigenfunctions are particularly unstable, and a generic perturbation is more likely to drift toward these directions. Notably, these eigenfunctions are induced by dilations and translations of the flow. If the perturbed RMCF drifts towards these directions, the perturbed MCF would only be shifted in the spacetime. We will use a \emph{centering map} to modulo the dilations and translations. Then we show that after this modulation, the most generic direction corresponds to the third eigenfunction direction.

	The paper is structured as follows. Section \ref{S:Spherical singularity} provides background information on spherical singularities. In Section \ref{S:Local dynamics and invariant cone}, we review and introduce the general dynamic tools and their applications to mean curvature flow. Section \ref{S:Stability of slow singularities} focuses on demonstrating the stability of slow singularities under perturbations. In Section \ref{S:Slow singularities are generic}, we show that it is always possible to perturb a spherical singularity into a slow spherical singularity. Additionally, the paper includes three appendices that present technical analyses related to translations/dilations and the transplantation of graphs.
	
	\subsection*{Acknowledgement}
	The authors want to thank Professor Bill Minicozzi for bringing their attention to the reference \cite{St1}. J.X. is supported by the grant NSFC (Significant project No.11790273 and No. 12271285) in China. J.X. thanks for the support of the New Cornerstone investigator program and the Xiaomi Foundation. The work is completed when J.X. is visiting Sustech International Center for Mathematics, to whom J.X. would like to thank its hospitality.

	\section{Spherical singularity}\label{S:Spherical singularity}
	In this section, we present some fundamental aspects concerning the spherical singularity of mean curvature flow. Let $\{\bM_\tau\}_{\tau\in I}$ denote a mean curvature flow (MCF), and we employ $\cM\subset\R^{n+1}\times\R$ to represent the spacetime track of this flow. We focus exclusively on the case where $n>1$.
	
	Given a spacetime point $X=(x,\tau)\in\R^{n+1}\times\R$, a rescaled mean curvature flow (RMCF) zooming in at $X$ is defined by
	\[M_t=e^{t/2}(\mathbf M_{\tau-e^{-t}}-x).\]
	
	We say a spacetime point $(x,\tau)$ is a spherical singularity if the RMCF $\{M_t\}_{t\in[0,\infty)}$ zooming in at $(x,\tau)$ smoothly converges to the sphere $\bS^n:= S^n(\sqrt{2n})$. Alternatively, this indicates the existence of a $T_0>0$ such that for $t>T_0$, $M_t$ can be expressed as a graph of the function $u(\cdot,t)$ over $\bS^n$, and $\|u(\cdot,t)\|_{C^\ell}\to 0$ as $t\to\infty$ for any $\ell\in\mathbb N$.
	
	The linearized operator of the RMCF over $\bS^n$ is defined as
	$
	L=\Delta+1,
	$
	where $\Delta$ denotes the Laplace operator. The RMCF equation for the graph function $u$ can be expressed as
	\[
	\pr_t u=Lu+\cQ(u,\nabla u,\nabla^2 u),
	\]
	where $\cQ(u,\nabla u,\nabla^2 u)$ represents a term that is  quadratically small in $u$. For a detailed calculation of the quadratic remainder term, we refer interested readers to \cite[Appendix]{CM4}.
	On the sphere $\bS^n$, the eigenfunctions of the operator $L$ are well-known to be the restrictions of homogeneous harmonic polynomials from $\R^{n+1}$. These eigenfunctions correspond to the eigenvalues $\lambda_k=-\left(\frac{k(k+n-1)}{2n}-1\right)$, where $k=0,1,2,\dots$. Notably, we have $\lambda_0=1$, $\lambda_1=1/2$, and $\lambda_2=-1/n$. For $k\geq 2$, the eigenvalues $\lambda_k$ are negative. Specifically, the eigenfunction associated with eigenvalue $1$ is the constant function, while the eigenfunctions corresponding to eigenvalue $-1/2$ are obtained by restricting the coordinate functions $x_i|_{\bS^{n}}$ for $i=1,2,\dots,n+1$.
	
	We would like to remind the readers that our convention for eigenvalues follows the order $\lambda_0>\lambda_1>\lambda_2>\cdots>\lambda_k>\cdots$, with $\lambda_k\to-\infty$ as $k\to\infty$. It is important to note that this convention differs from the conventions used by Colding-Minicozzi and Strehlke. We have adopted this specific convention as it aligns more favorably with our dynamical approach to the problem at hand.
	
	The RMCF can be viewed as a dynamical system when the initial hypersurface is close to a sphere. The investigation of the local dynamics of RMCF near a closed shrinker was initiated in \cite{CM4}. In our work \cite{SX1}, we establish a local stable/unstable manifold theorem for RMCF near any closed self-shrinker. 
	
	Strehlke \cite{St1} studied the local dynamics of RMCF near a sphere and obtained the following asymptotic result of RMCF.
	
	\begin{theorem}[Theorem 2.1 and (4) in \cite{St1}]\label{thm:Thm2.1inSt1}
		Suppose the RMCF $\{M_t\}$ smoothly converges to $\bS^n$ and it can be written as a graph of function $u(\cdot,t)$ over $\bS^n$ for $t>T_0$. Then either $u\equiv 0$ or there exists $k\geq 2$ and a nonzero homogeneous harmonic polynomial $P_k$ of degree $k$, and $\sigma>0$, such that we have as $t\to\infty$ 
		\begin{equation}
			u(y,t)=e^{\lambda_k t}P_k(y)|_{\mathbb S^n}+O(e^{(\lambda_k-\sigma)t}).
		\end{equation}
	\end{theorem}
	
	This theorem can be interpreted as an asymptotic result of level set flow.
	
	\begin{theorem}[Theorem 1.1 in \cite{St1}]
		Let $t$ be a solution to the level set equation on a smooth bounded convex domain $\Omega\subset\R^{n+1}$
		which attains its maximum $T$ at the origin. Then either $\Omega$ is a round ball and $t=T-|x|^2/(2n)$ for $x\in\Omega$,
		or there exists an integer $k\geq 2$ and a nonzero homogeneous harmonic polynomial $P_k$ of degree $k$ for which $t$
		has, at the origin, the asymptotic expansion
		\begin{equation}\label{eq:LSF expansion}
			t(x)=T-\frac{|x|^2}{2n}+|x|^{k(k-1)/n}P_k(x)+O(|x|^{\sigma+k+k(k-1)/n})
		\end{equation}
		for some $\sigma>0$.
	\end{theorem}
	
	Let us now discuss the regularity of the arrival time function $t(x)$. Huisken \cite{H3} demonstrated that for an arrival time function $t(x)$ with a single spherical singularity, the function $t$ belongs to the class $C^2$. Sesum \cite{Se}, on the other hand, constructed examples illustrating that, in general, the function $t$ does not belong to $C^3$. It is worth noting that there is a conjecture suggesting that if $t(x)=T-\frac{|x|^2}{2n}+O(|x|^N)$, then $t$ should lie in $C^{N-1}$. This is mentioned in the Remark following Theorem 1.1 in \cite{St1}. By examining the expansion given in equation \eqref{eq:LSF expansion}, we observe that when $n\geq2$ and $k=2$, the arrival time function $t(x)$ is not $C^3$ at $x=0$.
	
	Translating this observation back to the RMCF, we can state the following proposition:
	
	\begin{proposition}
		Suppose the RMCF $\{M_t\}_{t\in[0,\infty)}$ smoothly converges to $\mathbb S^n$ and it can be written as a graph of function $u(\cdot,t)$ over $\mathbb S^n$ for $t>T_0$. If there exists a nonzero homogeneous harmonic polynomial $P_2$ of degree $2$, and $\sigma>0$, such that
		\begin{equation}\label{eq:slowconv}
			u(y,t)=e^{- t/n}P_2(y)|_{\mathbb S^n}+O(e^{(-1/n-\sigma)t}),
		\end{equation}
		then the associated LSF function is not $C^3$.
	\end{proposition}
	
	Equation \eqref{eq:slowconv} describes the slowest possible convergence rate. We say the RMCF is \emph{slow} if the asymptotics is given by \eqref{eq:slowconv}, and the corresponding singularity is referred to as a \emph{slow singularity}.
	
	The connection between Theorem \ref{thm:main} and Theorem \ref{thm:main2} is now evident. The proof of Theorem \ref{thm:main2} is divided into two key components. The first part is based on a stability statement, implying that a slow singularity remains unchanged under a small initial perturbation. The second part focuses on a denseness result, stating that if a singularity is not slow, an arbitrarily small initial perturbation can make it slow. These two aspects will be addressed in Section \ref{S:Stability of slow singularities} and Section \ref{S:Slow singularities are generic}, respectively.
	
	\section{Local dynamics and invariant cones}\label{S:Local dynamics and invariant cone}
	In this section, we study the local dynamics near a sphere. Let $\cX$ be the space $H^1(\bS^n)$ equipped with the $Q$-norm: for any $u\in H^1(\bS^n)$,
	\[\|u\|_{Q}=\left(\int_{\bS^n}(|\nabla u(x)|^2+\Lambda u(x)^2 )d\cH^n(x)\right)^{1/2},\]
	where we choose $\Lambda$ to be a positive number greater than the leading eigenvalue of $L$. The $Q$-norm was initially introduced by Colding-Minicozzi in \cite{CM4}. It can be verified that the $Q$-norm is equivalent to the $H^1$-norm. For simplicity, we will denote $\|\cdot\|=\|\cdot \|_Q$ in the following discussions.
	
	The space $\cX$ admits a natural $L^2$-splitting, given by $\cX=\cX_-\oplus\cX_2\oplus\cX_+$, where $\cX_+$ is the subspace spanned by eigenfunctions with eigenvalues $1$ and $1/2$, $\cX_2$ is the subspace spanned by eigenfunctions with eigenvalue $\lambda_2=-1/n$, and $\cX_-$ is the $L^2$-orthogonal complement. We employ $\pi_+$, $\pi_2$, and $\pi_-$ to denote the orthogonal projections onto $\cX_+$, $\cX_2$, and $\cX_-$, respectively.
	
	Given $\kappa>0$, we define the cones
	\[
	\cK^+_\kappa=\{u=(u_-,u_2,u_+)\in \cX_-\oplus\cX_2\oplus\cX_+: \kappa\|u_-+u_2\|\leq \|u_+\|\},
	\]
	\[
	\cK_\kappa=\{u=(u_-,u_2,u_+)\in \cX_-\oplus\cX_2\oplus\cX_+: \kappa\|u_-\|\leq \|u_++u_2\|\}.
	\]
	Geometrically $\cK^+_\kappa$ (respectively $\cK_\kappa$) is a cone around $\cX_+$ (respectively $\cX_+\oplus \cX_2$). The value of $\kappa$ determines the opening angle of the cone, such that larger $\kappa$ results in a narrower cone. These cones are known as the \emph{invariant cones} as they remain unchanged under the RMCF. In fact, these cones can be refined over time, becoming progressively narrower and narrower.
	
	\begin{theorem}\label{thm:invariant cone}
		For any $c>0$, there exist $\delta_0>0$ and $\bar\kappa=\bar\kappa(\dt_0)>0$ with the following significance. Suppose $u(t),v(t):\ \bS^n\times [0,m]\to \R,\ m\in \mathbb N,$ are graphical functions of solutions to the RMCF equation, then whenever $\|u(t)\|_{C^{2,\alpha}}\leq\delta_0$, $\|v(t)\|_{C^{2,\alpha}}\leq \delta_0$ for all $t\in [0,m]$, we have
		\begin{enumerate}
			\item If $u(0)-v(0)\in \cK^+_{\kappa}$ for $\kappa\in[c,\bar\kappa]$, then $u(m)-v(m)\in \cK^+_{e^{\frac{m}{4}}\kappa}\cup \cK^+_{\bar\kappa}$.
			\item If $u(0)-v(0)\in \cK_{\kappa}$ for $\kappa\in[c,\bar\kappa]$, then $u(m)-v(m)\in \cK_{e^{\frac{m}{n}}\kappa}\cup \cK_{\bar\kappa}$.
		\end{enumerate}
		Moreover, we have $\bar\kappa\to\infty$ as $\dt_0\to 0$.
	\end{theorem}
	
	We need the following crucial estimate in \cite{CM4}. This estimate implies that for two graphs over $\bS^n$ that are close to each other, after evolving for a time of $1$, they are still close to each other. In the subsequent expressions, $e^L$ denotes the time-one evolution operator associated with the heat equation $\partial_t u=Lu$.
	
	\begin{proposition}[Proposition 5.2 in \cite{CM4}]\label{prop:CM}
		There exist $\delta_1>0$, $\eps>0$ and $C>0$ with the following significance. Suppose $u(t)$ and $v(t)$ are graphical functions of solutions to the RMCF over the limiting closed shrinker for $t\in[0,1]$ and $\|u(0)\|_{C^{2,\alpha}}\leq \delta<\delta_1$, $\|v(0)\|_{C^{2,\alpha}}\leq \delta<\delta_1$. Then
		\[
		\|(u(1)-v(1))-e^L(u(0)-v(0))\|^2\leq C\delta^\eps\|u(0)-v(0)\|_{L^2}^2.
		\]
	\end{proposition}
	
	\begin{proof}[Proof of Theorem \ref{thm:invariant cone}]
		We shall now prove item (1), and the proof for item (2) follows similarly. First, we focus on the case $m=1$. Let $\eps_0$ be a sufficiently small number to be determined. From Proposition \ref{prop:CM}, by choosing $\delta_0<\delta_1$ to be sufficiently small, we obtain the following:
		\begin{equation}\label{eq:CM_est}
			\|(u(1)-v(1))-e^L(u(0)-v(0))\|^2\leq \eps_0^2\|u(0)-v(0)\|_{L^2}^2.
		\end{equation}
		If $u(0)-v(0)\in\cK^+_{\kappa}$, we have $\kappa\|(\pi_-+\pi_2)(u(0)-v(0))\|\leq \|\pi_+(u(0)-v(0))\|$. Since $e^L$ is a linear operator that simply scales the eigenvectors by the exponential of their eigenvalues, we have the following
		\[
		\|(\pi_-+\pi_2)(e^Lu(0)-e^Lv(0))\|\leq e^{-1/n} \|(\pi_-+\pi_2)(u(0)-v(0))\|,
		\]
		\[
		\|\pi_+(e^Lu(0)-e^Lv(0))\|\geq e^{1/2} \|\pi_+(u(0)-v(0))\|.
		\]
		
		Therefore, $e^L(u(0)-v(0))\in\cK^+_{e^{1/2}\kappa}$. Let $\eps_0$ in \eqref{eq:CM_est} sufficiently small, using the fact that $L^2$-norm is bounded by $H^1$-norm, we have 
		\[
		\begin{split}
			\|(\pi_-+\pi_2)(u(1)-v(1))\|
			\leq& 
			e^{-1/n} \|(\pi_-+\pi_2)(u(0)-v(0))\|+\eps_0\|u(0)-v(0)\|
			\\
			\leq&
			\left(e^{-1/n}\frac{1}{1+\kappa}+\eps_0\right)\|u(0)-v(0)\|
			,
		\end{split}
		\]
		\[
		\begin{split}
			\|\pi_+(u(1)-v(1))\|
			\geq&
			e^{1/2}\|\pi_+(u(0)-v(0))\|-\eps_0\|u(0)-v(0)\|
			\\
			\geq&
			\left(e^{1/2}\frac{\kappa}{1+\kappa}
			-\eps_0
			\right)\|u(0)-v(0)\|
			.
		\end{split}
		\]
		
		As a consequence,
		\[
		\|(\pi_-+\pi_2)(u(1)-v(1))\|
		\leq 
		\frac{e^{-1/n}+(1+\kappa)\eps_0}{e^{1/2}\kappa-(1+\kappa)\eps_0}\|\pi_+(u(1)-v(1))\|.
		\]
		Whenever $\eps_0\leq \frac{e^{1/4+1/n}\kappa}{(1+e^{1/4}\kappa)(1+\kappa)}$ or $e^{1/4}\kappa\geq\bar\kappa$, we have $u(1)-v(1)\in \cK^+_{e^{1/4}\kappa}$. Iterating the above process $m$ times gives the desired conclusion. 
		
		In the above discussion, we only require $\eps_0\leq \frac{e^{1/4+1/n}e^{\ell/4}\kappa}{(1+e^{1/4}e^{\ell/4}\kappa)(1+e^{\ell/4}\kappa)}$ holds for all $\ell\in[0,m]$, or $e^{\ell/4}\kappa\geq\bar\kappa$ for some $\ell\in[0,m]$. In particular, we can choose $\bar{\kappa}$ such that $\frac{\bar{\kappa}}{(1+2\bar\kappa)^2}\geq \eps_0\approx C\delta_0^\eps/2$ to make the proof works. This implies that if $\delta_0\to 0$, we can allow $\bar\kappa\to\infty$.
	\end{proof}
	
	Furthermore, we emphasize that Proposition \ref{prop:CM} can be viewed as an indication that the solution to the linearized variational equation approximates the RMCF equation in the $H^1$-norm. This approximation extends to the H\"older norm as well. Inspired by \cite[Section 4]{CM4}, we present the following proposition:
	
	\begin{proposition}[Proposition 3.2 in \cite{SX1}]\label{prop:SX1prop3.2}
		Given an RMCF $\{M_t\}_{t\in[0,\infty)}$ converging to a closed self-shrinker $\Sigma$ and given $T>0$, there exist $\delta_0>0$, $\sigma_0>0$ and $C>0$ depending on $T$, with the following significance: if $v_0\in C^{2,\alpha}(M_0)$ satisfies $\|v_0\|_{C^{2,\alpha}}\leq \delta<\delta_0$, suppose $v$ satisfies the RMCF equation with initial data $v_0$, and $v^*$ satisfies the linearized equation $\pr_t v^*=L_{M_t}v^*$ with initial data $v_0$. Then for all $t\in[0,T]$,
		\[\|(v-v^*)(\cdot,t)\|_{C^{2,\alpha}}\leq C\delta^{1+\sigma_0}.\]
	\end{proposition}

	Lastly, to verify the assumptions of the above Theorem \ref{thm:invariant cone}, we require a consequence of the \L ojasiewich-Simon inequality established by Schulze \cite{Sc}. In cases where the perturbed RMCF converges to the same self-shrinker, we aim to demonstrate that these two RMCFs remain close to each other throughout the process, which allows us to apply the invariant cone argument. The precise statement is as follows:
	
	\begin{proposition}\label{prop:LS}
		Suppose $\{M_t\}_{t\in[0,\infty)}$ is an RMCF converging to a closed self-shrinker $\Sigma$. Given $\eps_0>0$, there exists $\delta_1>0$ with the following significance. For any $v_0\in C^{2,\alpha}(M_0)$ with $\|v_0\|_{C^{2,\alpha}}<\delta_1$, suppose $\widetilde{M_t}$ is the RMCF starting from $\{x+v_0(x)\bn(x)\ |\ x\in M_0\}$, and converging to the same $\Sigma$, then $\widetilde{M_t}$ can be written as a graph of function $v(\cdot,t)$ over $M_t$ for all $t\geq 0$ satisfying $\|v(\cdot, t)\|_{C^{2,\alpha}}\leq\eps_0$.
	\end{proposition}
	
	\begin{proof}
		We will choose some $T_0>0$, and prove the closeness of $\widetilde{M_t}$ and $M_t$ for the two cases $t<T_0$ and $t\geq T_0$ separately.
		
		By the proof of \cite[Theorem 1.1]{Sc}, there exists $\sigma_0>0$ with the following significance: suppose $N_t$ is an RMCF converging to $\Sigma$, and $N_0$ is a graph of function $w(0)$ over $\Sigma$ with $\|w(0)\|_{C^{2,\alpha}}<\sigma_0$, then $N_t$ is a graph of function $w(t)$ over $\Sigma$, with $\|w(t)\|_{C^3}\leq c_3 (1+t)^{-\alpha_m}$ as $t\to\infty$, for some constants $c_3>0$ and $\alpha_m>0$. We may further assume $\sigma_0<\eps_0$ where $\eps_0$ is stated in the proposition, and we suppose $\sigma_0$ is chosen so small such that we can apply the transplantation Theorem \ref{Thm:Appendix closeness of graphs} where $\eps$ in Theorem \ref{Thm:Appendix closeness of graphs} is smaller than $1/2$. As a consequence, there exists $T>0$ such that for $t\geq T$, $M_t$ is a graph of function $u(t)$ over $\Sigma$, such that $\|u(t)\|_{C^{3}}<c_3 (1+t-T)^{-\alpha_m}$. Then we choose $T_0>T$ such that $c_3 (1+T_0-T)^{-\alpha_m}<\sigma_0/4$.
		
		For this $T_0$, we apply Proposition \ref{prop:SX1prop3.2} to see that for $\delta_1<\delta_0$ sufficiently small, if $\|v_0\|_{C^{2,\alpha}}<\delta_1$, for $t\in[0,T_0]$,
		\[
		\|v(t)\|_{C^{2,\alpha}}\leq C\delta_1^{1+\sigma_0}+C'\delta_1,
		\]
		where $C'\delta_1$ comes from $\|v^*(t)\|_{C^{2,\alpha}}$ and we notice that $v^*$ is a solution to the linear equation in Proposition \ref{prop:SX1prop3.2}. In particular, if we choose a sufficiently small $\delta_1$, we have $\|v(T)\|_{C^{2,\alpha}}\leq \sigma_0/4$.
		
		Then we can write $\widetilde{M_T}$ as a graph of the function $\tilde v(T)$ over $\Sigma$, and we have $\|\tilde v(T)\|_{C^{2,\alpha}(\Sigma)}\leq \sigma_0$ from the transplantation Theorem \ref{Thm:Appendix closeness of graphs}. This implies that $\|\tilde v(t)\|_{C^{3}}<c_3 (1+t-T)^{-\alpha_m}$. In particular, $\|\tilde v(T_0)\|_{C^{2,\alpha}}<\sigma_0/4$, which also implies that $\|v(t)\|_{C^{2,\alpha}}<\sigma_0<\eps_0$ for $t\geq T_0$. Combining all the discussions above we conclude the proof.
	\end{proof}
	
	Throughout the rest of the paper, we will select the perturbation to be sufficiently small, such that its $C^{2,\alpha}$-norm is smaller than $\delta_0$ in the preceding proposition.

	\section{The stability of slow singularities}\label{S:Stability of slow singularities}
	
	In this section, we prove the stability of slow singularities under small perturbations in the initial data. Specifically, we prove the following theorem:
	
	\begin{theorem}[The stability theorem]\label{thm:stability}
		Suppose $\{\bM_\tau\}_{\tau\in[-1,0)}$ is an MCF with a unique spherical singularity at the spacetime point $(0,0)$, which is slow. Then there exists $\eps_0>0$ such that for $u\in C^{2,\alpha}(\bM_0)$ with $\|u\|_{C^{2,\alpha}}<\eps_0$, the MCF starting from $\{x+u(x)\bn(x)\ |\ x\in \bM_0\}$ also has a unique spherical singularity, and the singularity is slow.
	\end{theorem}

	In the rest of this section, we shall focus on a fixed MCF $\{\bM_\tau\}_{\tau\in[-1,0)}$ with a slow singularity, along with its corresponding RMCF $\{M_t\}_{t\in[0,\infty)}$.
	
	The local stability of spherical singularities was proved by Huisken. In \cite{H1}, Huisken showed that any convex mean curvature flow in $\R^{n+1}$ must converge to a point, which is a spherical singularity.
	
	If the perturbation applied to the initial data is sufficiently small, the perturbed flow will remain close to the original flow for an extended period. The following proposition is a consequence of Proposition \ref{prop:SX1prop3.2}.
	
	\begin{proposition}\label{prop:smooth-dependence-on-initial-data-1}
		For any RMCF $\{M_t\}_{t\in[0,\infty)}$, any $T_0>0$ and $\eta>0$, there exists $\eps_1>0$ such that for any $u_0\in C^{2,\alpha}(M_0)$ with $\|u_0\|_{C^{2,\alpha}(M_0)}<\eps_1$, the perturbed RMCF starting from $\{x+u_0(x)\bn(x)\ |\ x\in M_0\}$ can be written as a graph of function $u(\cdot,t)$ over $M_t$ for $t\in[0,T_0]$, and $\|u(\cdot,T_0)\|_{C^{2,\alpha}(M_{T_0})}<\eta$.
	\end{proposition}
	
	The same argument applies to the MCF if we scale it back (in fact, the proof in \cite{CM4} uses an argument by Ecker-Huisken for MCF in \cite{EH}, where rescaling the RMCF to an MCF is necessary). Consequently, if we perturb the initial data of an MCF that converges to a spherical singularity, the perturbed MCF would eventually become convex after a sufficiently long time. Therefore, the perturbed flow must also have a spherical singularity.
	
	The following proposition is a result from geometric measure theory, which can be found in \cite[Section 10.1]{Wh4}.
	
	\begin{proposition}\label{prop:trans-dilat-sing}
		Let  $\{\cM_i\}_{i=1}^\infty$ be a sequence of MCF spacetime with a single spherical singularity $X_i\in\cM_i$, and $\cM$ be an MCF spacetime with a single spherical singularity $X\in\cM$. Suppose $\cM_i\to\cM$ in the sense of measure as $i\to\infty$. Then we have $X_i\to X$.
	\end{proposition}
	
	\begin{proof}
		Since $\cM_i\to\cM$, we can assume that all $\cM_i$ and $\cM$ are bounded by a sufficiently large ball. Then, the avoidance principle of mean curvature flow implies that the points $X_i$ lie in a compact subset of spacetime. Therefore, a subsequence of ${X_i}$ converges to a limit point $X$. The upper-semi-continuity property of Gaussian density, as shown in the equation before Proposition 10.6 in \cite{Wh4}, demonstrates that $X$ is a spherical singularity of $\cM$. According to Huisken's theorem, $X$ must be the unique singularity of $\cM$. Hence, the entire sequence ${X_i}$ converges to this unique limit point $X$.
	\end{proof}
	
	Now we are going to study the influence of perturbation on singularity. A defect of RMCF is that it can only detect a single spacetime point, and if the singularity is moved in the spacetime, the RMCF fails to provide relevant information. To avoid this defect, we introduce a map called \emph{centering map}.
	
	\begin{definition}[Centering map]
		Let $\{M_t\}$ be an RMCF that converging to a sphere and $u_0\in C^{2,\alpha}(M_0)$ be a function with sufficiently small $C^{2,\alpha}$ norm. A composition of a translation and dilation $\cC_{u_0}$ of $\R^{n+1}$ is called a \emph{centering map} if the MCF starting from $\cC_{u_0}(\{x+u_0(x)\bn(x)\ |\ x\in M_0\})$ has a spherical singularity at the same spacetime position as the MCF corresponding to $\{M_t\}$.
	\end{definition}
	
	In other words, the centering map $\cC_u$ enables us to maintain the spacetime position of the spherical singularity. We equip $\cC_u$ with the norm in the space of translations and dilations.
	
	By employing a contradiction argument along with Proposition \ref{prop:trans-dilat-sing}, we can derive a useful control on the centering map.
	
	\begin{lemma}\label{lem:controlT}
		Given $\eta_1>0$, there exists $\eps_2>0$, such that when $\|u\|_{C^{2,\alpha}(M_0)}<\eps_2$, $\|\cC_u\|<\eta_1$.
	\end{lemma}

	With this Lemma, we can incorporate the centering map to Proposition \ref{prop:smooth-dependence-on-initial-data-1}. Now we prove Theorem \ref{thm:stability}.
	
	\begin{proof}[Proof of Theorem \ref{thm:stability}]
		Suppose $M_t$ is a graph of function $u(t)$ over $\bS^n$ when $t$ is sufficiently large. Because $M_t$ has a slow singularity, for any $\kappa>0$, $u(t)\in\cK_\kappa$ when $t$ is sufficiently large. Let us fix $T>0$ such that $u(T)\in\cK_\kappa$. Suppose $\|u(T)\|=\eta$.
		Suppose $\eps_0$ is sufficiently small, to be determined later. By the well-posedness of RMCF, for $\|v_0\|_{C^{2,\alpha}}<\eps_0$, the perturbed RMCF $\widetilde M_t$ starting from $\cC_{u_0}(\{x+u_0(x)\bn(x)\ |\ x\in M_0\})$ can be written as a graph of function $v$ over $M_t$ for all $t\in[0,T]$, and $\|v(T)\|_{C^{2,\alpha}}\leq \eta'$, $\eta'$ to be determined.
		We choose $\eps_0$ so that $\eta'$ is sufficiently small, compared with $\eta$. As a consequence, we get $\|v(T)\|\leq \eta/C$, where $C$ is sufficiently large, such that if we write $\widetilde{M}_T$ as a graph of function $w$ over $\bS^n$, then $w(T)\in \cK_{\kappa/2}$ (see Theorem \ref{Thm:Appendix closeness of graphs-Q}).

		If we choose $\kappa$ sufficiently large at the beginning, Theorem \ref{thm:invariant cone} implies that $w(t)\in \cK_{\kappa/2}$ for all $t>T$. This implies that the dominant term of $w$ is the eigenfunctions in $\cX_2\oplus\cX_+$. Moreover, for all $\eps$ there exists $T$ such that for all $t>T$, we have $\|\pi_2w(t)\|\geq (1-\eps)\|w(t)\|$. Then Theorem \ref{thm:Thm2.1inSt1} implies that $w(t)$ characterizes a slow singularity. Indeed, we have proved that the projection of $w(t)$ to $\cX_2\oplus\cX_+$ dominates that of $\cX_-$. It is enough to show that $\pi_2 w(t)$ dominates also $\pi_+w(t)$. Suppose $w\in \cK_{\eps}^+$ for some small $\eps$, then Theorem \ref{thm:invariant cone} implies that $w(t)\in \cK_{\eps}^+$ for all future time and the mode $\pi_+w(t)$ grows exponentially like $e^{t/2}$ contradicting to the fact that the perturbed RMCF converges to $\bS^n$ after applying the centering map. This completes the proof.
		
	\end{proof}

	\section{Slow singularities are generic}\label{S:Slow singularities are generic}
	
	In this section, we demonstrate that slow singularities are generic. In the previous section, we established that the set of initial data that leads to a slow singularity is open. Now, we aim to show that this set is also dense, thereby establishing its genericity.
	
	\begin{theorem}[The denseness theorem]\label{thm:denseness}
		Suppose $\{\bM_\tau\}_{\tau\in[0,T)}$ is an MCF with a unique spherical singularity at the spacetime point $(0,0)$. Then for any $\eps_0>0$ there exists $u\in C^{2,\alpha}(\bM_0)$ with $\|u\|_{C^{2,\alpha}}<\eps_0$, such that the MCF starting from $\{x+u(x)\bn(x)\ |\ x\in \bM_0\}$ also has a unique spherical singularity that is slow.
	\end{theorem}

	If the MCF $\{\bM_\tau\}_{\tau\in[0,T)}$ already has a slow singularity, then we can simply choose $u=0$ in Theorem \ref{thm:stability} to ensure the singularity remains slow. Therefore, the more intriguing case is when $\{\bM_\tau\}_{\tau\in[0,T)}$ is a mean curvature flow with a singularity that is not slow.
	
	By applying Theorem \ref{thm:Thm2.1inSt1}, we can assume that for sufficiently large $t$, the RMCF $M_t$ can be represented as a graph of a function $u(\cdot,t)$ over $\mathbb{S}^n$, and there exist $k\geq 3$ and a degree $k$ homogeneous harmonic polynomial $P_k$, such that as $t\to\infty$,
	\[u(y,t)=e^{\lambda_k t}P_k(y)|_{\bS^n}+O(e^{(\lambda_k-\sigma)t}).\]
	
	First, we aim to demonstrate that we can select initial data in such a way that the solution to the variational equation can approximate the slow direction. This connection is crucial for understanding the impact of initial perturbations on perturbations near the singularity. The lemma below, which was essentially established by Herrero-Vel\'azquez in \cite{HV}, plays a significant role in our analysis. It should be noted that the solution to the variational equation $\partial_t v=L_{M_t}v$ with initial data $v_0$ can be expressed as $v(t)=\mathcal{T}(t,0)v_0$.
	
	\begin{lemma}[Lemma 5.2 in \cite{HV}]\label{LmHV}
		For any $t>0$ the map $\mathcal T(t,0):L^2(M_0)\to L^2(M_t)$ has dense image.
	\end{lemma}
	
	We prove this Lemma in Appendix \ref{AS:Denseness of image of the fundamental solution} using the idea of Herrero-Vel\'azquez in \cite{HV}. The next lemma raises the $L^2$-denseness to $Q$-denseness. We use $Q(M_t)$ to denote  $H^1(M_t)$ equipped with $Q$-norm. We prove it in Appendix \ref{AS:Denseness of image of the fundamental solution}.
	
	\begin{lemma}\label{lem:dense}
		For any $t>0$, the map $\mathcal T(t,0):\ L^2(M_0)\to Q(M_t)$ has dense image.
	\end{lemma}
	
	For sufficiently large $t$, we write $M_t$ as a graph of a function $u(t)$ over $\bS^n$, pull back functions on $M_t$ to functions on $\mathbb S^n$ and identify the function spaces $H^1(M_t)$ and $H^1(\mathbb S^n)$ (we refer to this process as ``transplantation'' as discussed in \cite{SX1}, see Appendix \ref{AS:Transplantation}). Consequently, at least on the linear level, we can select an initial condition in such a way that at time $T$, the solution is near the eigenfunction direction corresponding to eigenvalue $\lambda_2=-1/n$. To be more precise, for any $T>0$ and $\epsilon>0$, it is possible to choose an initial data $v_0$ such that the solution $v(\cdot,t)$ to $\partial_t v=L_{M_t}v$ with the initial data $v_0$ satisfies $\|\pi_2 v(T)\|>(1-\epsilon/4)\|v(T)\|$.
	
	Next, we consider a perturbed RMCF $\widetilde{M}_t$ with an initial condition $\{x+\bar{\epsilon} v_0(x)\mathbf{n}(x)\ |\ x\in M_0\}$, where $\bar{\epsilon}$ is taken to be sufficiently small. According to Proposition \ref{prop:SX1prop3.2}, over the time interval $[0,T]$, the perturbed flow $\widetilde{M}_t$ can still be described as a graph of a function $u(t):\ M_t\to \R$, and we obtain $\|\pi_2 u(T)\|>(1-\epsilon/2)\|u(T)\|.$
	
	Now we can view the perturbation applied near the singularity. However, the perturbed MCF $\widetilde{\mathbf{M}}_\tau$ corresponding to $\widetilde{M}_t$ might not have the singularity at the spacetime point $(0,0)$. To address this, we apply the centering map to obtain a flow $\widehat{\mathbf{M}}_t: =\cC_{\bar{\epsilon} v_0}\widetilde{\mathbf{M}}_t$ that now has the singularity at the spacetime point $(0,0)$.
	
	The following lemma controls the size of the centering map.
	
	\begin{lemma}\label{lem:consequence of dense}
		Let $\eps>0$ be a small number, $T$ be a large time and $u:\ M_t\to \R$ be as above. Suppose $\|\pi_2 u(T)\|>(1-\eps)\|u(T)\|$. Then  after applying the centering map the flow $\widehat{\mathbf M}_t: =\cC_{\bar\eps v_0}\widetilde{\mathbf M}_t$ can be written as a graph of function $w(\cdot,t), t\in[0,\infty)$ over $M_t$, satisfying $\|\pi_2 w(T)\|>(1-2\eps)\|w(T)\|$.
	\end{lemma}
	
	We postpone the proof in Appendix \ref{AS:Action of translations and dilations}. Here we simply sketch the idea: because $\|\pi_2u(T)\|$ has a  large proportion in $\|u(T)\|$, the norm $\|\pi_+ u(T)\|$ has a small proportion. The centering map is dedicated to modulo the $\cX_+$ part, which is also small. Therefore, after applying the centering map, $\|\pi_2w(T)\|$ still dominates.
	
	\begin{proof}[Proof of Theorem \ref{thm:denseness}]
		Given $\eps_0$, let us choose $T_0$ sufficiently large, to be determined later.
		We  apply  Lemma \ref{lem:dense} to get  $v_0$ as an initial condition for the variational equation $\partial_t v^*=L_{M_t}v^*$ such that for $T$ sufficiently large, we have that $$\|\pi_2v^*(T)\|>(1-\eps/4) \|v^*(T)\|.$$
		Lemma \ref{lem:dense} gives only $v_0\in L^2$. Since $C^{r}$ is dense in $L^2$, we can assume $v_0$ is as smooth as we wish.
		
		Here $\eps$ is a small number to be determined. Now we choose $\overline{\eps}<\eps_0$ very small, such that $C\overline{\eps}^{\sigma_0}<\frac{1}{10}\eps\|v^*(T)\|$. Here $C$ and $\sigma_0$ are given in Proposition \ref{prop:SX1prop3.2}. Then Proposition \ref{prop:SX1prop3.2} implies that the perturbed RMCF $\widetilde{M}_t$ starting from $\overline{\eps}v_0$ can be written as a graph of function $v(t)$ over $M_t$ for $t\in[1,T]$, and $$\|\pi_2v(T)\|>(1-1/2\eps) \|v(T)\|.$$
		Suppose the perturbed RMCF $\widetilde{M_{T}}$ is a graph of function $w^*(T)$ over $\bS^n$. Then if at the beginning we choose $T$ sufficiently large, using Theorem \ref{Thm:Appendix closeness of graphs} we can show that  $$\|\pi_2 (w^*(T)-f(T))\|>(1-\eps) \|(w^*(T)-f(T))\|,$$
		where $f(t)$ is the graphical function of the unperturbed flow $M_t$ over $\bS^n$ for $t$ large.
		We next apply the centering map to translate the perturbed MCF $\widetilde{\mathbf M}_t$ corresponding to the RMCF $\widetilde{M}_{t}$.  Lemma \ref{lem:consequence of dense} implies that after the centering map, the RMCF $\widehat{M}_t$ corresponding to the centered MCF $\widehat{\mathbf M}_t: =\cC_{\bar\eps v_0}\widetilde{\mathbf M}_t$ can be written as a graph of function $w(t)$ over $\bS^n$ when $t$ is sufficiently large, and $$\|\pi_2(w(T)-f(T))\|\geq (1-2\eps)\|\pi_2(w(T)-f(T))\|.$$
		
		Then we have $w(T)-f(T)\in \cK_\kappa$ for some large $\kappa$. Theorem \ref{thm:invariant cone} together with Proposition \ref{prop:LS} imply that $w(t)-f(t)\in \cK_{\kappa}$ for all $t>T$. In particular, arguing similarly as the end of the proof of Theorem \ref{thm:stability}, this implies that $w$ has the dominant term in $\cX_2$, i.e. $w$ characterizes a slow singularity.  
	\end{proof}

	\appendix
	\section{Action of translations and dilations}\label{AS:Action of translations and dilations}
	
	In Euclidean space $\mathbb{R}^{n+1}$, the translations form a group denoted by $\mathbb{R}^{n+1}$, and the dilations form a semigroup denoted by $\mathbb{R}_+$. Each of these spaces carries a natural (Euclidean) norm. Let $(U,\alpha)$ be an element in $\mathbb{R}^{n+1}\times\mathbb{R}_+$. For any point or hypersurface, $(U,\alpha)$ acts by first dilating the point/hypersurface by $\alpha$ and then translating it by the vector $U$.
	Suppose $\Sigma$ is a hypersurface in $\mathbb{R}^{n+1}$ and $\Gamma$ is a graph of the function $u$ over $\Sigma$, which means $\Gamma = {x+u(x)\mathbf{n}(x) : x \in \Sigma}$, where $\mathbf{n}$ is the unit normal vector field on $\Sigma$. Let $|U|=1$. We want to understand how the action of $(\beta U, \alpha)$ on $\Gamma$ would change this graph. Suppose $(\beta U, \alpha)(\Gamma) = \widetilde{\Gamma}$. We assume that $|(\alpha U, \alpha)|$ is sufficiently small such that $\widetilde{\Gamma}$ is also the graph of a function $w$ over $\Sigma$. In other words, we have
	$$\Gamma=\{x+u(x)\bn(x):x\in\Sigma\}, \mathrm{\ and\ } \widetilde{\Gamma}=\{x+w(x)\bn(x):x\in\Sigma\}.$$ The action of $(U,\alpha)$ indicates that
	\[
	\{\alpha(x+u(x)\bn(x))+U:x\in\Sigma\}=\{x+w(x)\bn(x):x\in\Sigma\}.
	\]
	
	We will be interested in the case that $\Sigma$ is a sphere. In this case, we can assume $x=\bn(x)$, and the second fundamental form is the identity. Then $\alpha(r+u(x_0))\bn(x_0)+\beta U=x_{\alpha,\beta}+w_{\alpha,\beta}(x_{\alpha,\beta})\bn(x_{\alpha,\beta})$. 
	We have the following control on $Q$-norm.
	\begin{lemma}\label{lem:A-Taylor}
		There exists $\eps_0>0$ with the following significance. When $\|u\|_{C^1}\leq \eps_0$, $|\alpha-1|\leq\eps_0$ and $\beta<\eps_0$, we have
		\begin{equation}
			\left\|w_{\alpha,\beta}-u-\left[(\alpha-1)+
			\beta\langle U,x\rangle\right]\right\|_{Q}\leq C(|\alpha-1|^2+|\beta|^2+\|u\|_{C^1}(|\alpha-1|+|\beta|)),
		\end{equation}

		where $C$ is a uniform constant.
	\end{lemma}
	
	\begin{proof}

		It is clear that $x_{1,0}=x$ and $w_{1,0}(x)=u(x_0)$. Taking derivative of $\alpha$ and $U$ we obtain the following: at $\alpha=1,\beta=0$,
		\begin{equation}
			\pr_{\alpha}x_{\alpha,\beta}=0, \
			\pr_\beta x_{\alpha,\beta}=\frac{U^\top}{1+u}.
		\end{equation}
		and
		\begin{equation}
			\pr_{\alpha} w_{\alpha,\beta}
			=
			1+u,\
			\pr_{\beta} w_{\alpha,\beta} =\langle U,x_0\rangle- \frac{\langle \nabla u,U\rangle}{1+u}.
		\end{equation}
		
		Then we have
		\begin{equation}\label{eq:A-Taylor}
			w_{\alpha,\beta}=u+(1+u)(\alpha-1)+
			\left(\langle U,x_0\rangle- \frac{\langle \nabla u,U\rangle}{1+u}\right)\beta
			+O((\alpha-1)^2+\beta^2).
		\end{equation}
		
		Similarly, taking the gradient we get
		\[
		\nabla w(x)
		= \left(\frac{D x_{\alpha,\beta}}{Dx}\right)^{-1} \alpha\nabla u(x),
		\]
		where $\frac{D x_{\alpha,\beta}}{Dx}$ is the differential of $x_{\alpha,\beta}$ at $\alpha=1$, $\beta=0$. From the local expansion
		\[
		\frac{D x_{\alpha,\beta}}{Dx}=\mathrm{Id}
		+\beta \frac{\nabla U^\top}{1+u}
		-\beta \frac{U\otimes\nabla u}{(1+u)^2} +O((\alpha-1)^2+\beta^2),
		\]
		we can see that
		\begin{equation}\label{eq:A-Taylor-2}
			\begin{split}
				\nabla w(x)
				=&
				\nabla u(x) +(\alpha-1)\nabla u(x)+
				\beta\nabla\langle U,x\rangle +
				\\
				&+ O(|\alpha-1|^2+|\beta|^2+\|u\|_{C^1}(|\alpha-1|+|\beta|)).
			\end{split}
		\end{equation}
		
		Pointwise we have the Taylor expansion \eqref{eq:A-Taylor} and \eqref{eq:A-Taylor-2}. Then we integrate the square and use Cauchy-Schwarz inequality to get the desired inequality.
	\end{proof}
	
	Next, we prove Lemma \ref{lem:consequence of dense}. The main idea is to show that if the $\cX_2$ part of a function $u$ has a large proportion in $\|u\|$, then the $\cX_+$ part has a small proportion, hence the centering map is also small. Otherwise, the centering map will enhance the $\cX_+$ part a lot, then the invariant cone argument (Theorem \ref{thm:invariant cone}) shows that the graph RMCF grows exponentially, which is a contradiction to the definition of a centering map.
	
	Let us recall the setting. Suppose $M_t$ is a graph of function $f$ over $\bS^n$, and suppose $\widetilde{M_t}$ is a graph of function $u$ over $\bS^n$. Suppose both $\|f\|_{C^2}$ and $\|u\|_{C^2}$ are sufficiently small. Throughout the rest of this section, $\|\cdot\|$ is the $Q$-norm. Suppose $\|\pi_2(u-f)\|\geq (1-\eps)\|u-f\|$ for some sufficiently small $\eps$. Now we apply the centering map to $\widetilde{M_t}$, and suppose after the centering map $\cC(\widetilde{M_t})$ is a graph of a function $w(t)$ over $\bS^n$.
	
	\begin{lemma}[Lemma \ref{lem:consequence of dense}]
		In the above setting, we have $$\|\pi_2(w-f)\|\geq (1-2\eps)\|w-f\|.$$
	\end{lemma}
	
	\begin{proof}
		By triangle inequality, $\|\pi_2(w-f)\|\geq \|\pi_2(u-f)\|-\|\pi_2(u-w)\|_{L^2}$, $\|w-f\|\leq \|(u-f)\|+\|u-w\|$. Thus
		\[
		\begin{split}
			\|\pi_2(w-f)\|
			\geq&
			(1-\eps)(\|w-f\|-\|u-w\|)-\|\pi_2(u-w)\|
			\\
			\geq&
			(1-\eps)\|w-f\|-2\|u-w\|.
		\end{split}
		\]
		Then we only need to show $\|u-w\|\leq \frac{\eps}{2}\|w-f\|$.
		Suppose the centering map is given by $(\alpha,\beta U)$. Lemma \ref{lem:A-Taylor} implies that when $|\alpha-1|$ and $|\beta|$ and $\|u\|_{C^1}$ are sufficiently small,
		\[
		\|\pi_+(u-w)\|\geq (|\alpha-1|+|\beta|)(1-2\eps_0).
		\]
		Then we have
		\[
		\|\pi_+(w-f)\|\geq (|\alpha-1|+|\beta|)(1-2\eps_0)
		-
		\|\pi_+ (u-f)\|.
		\]
		Because $w$ is the graph of a hypersurface after the centering map, $\|\pi_+(w-f)\|\leq \eps\|(w-f)\|$, since otherwise, Theorem \ref{thm:invariant cone} would imply that the cone $\cK^+_\eps$ is preserved under the dynamics and $\pi_+(w-f)$ grows exponentially hence changes the spacetime location of the singularity, which is a contradiction to the definition of the centering map. So we get the estimate
		\[|\alpha-1|+|\beta|\leq \frac{1}{1-2\eps_0}
		(\eps\|(w-f)\|+\|\pi_+(u-f)\|).
		\]
		Lemma \ref{lem:A-Taylor} also implies that
		\[
		\|(u-w)\|\leq(1+2\eps_0) (|\alpha-1|+|\beta|).
		\]
		From $\|\pi_2(u-f)\|\geq (1-\eps)\|u-f\|$, we also know that $\|\pi_+(u-f)\|\leq \frac{\eps}{1-\eps}\|u-f\|$.
		Combining all the ingredients above shows that $\|\pi_+(w-f)\|\leq \eps\|(w-f)\|$.
	\end{proof}
	\section{Denseness of image of the fundamental solution}\label{AS:Denseness of image of the fundamental solution}
	
	Let $M_t, t\in[0,T]$ be an RMCF. Let $\mathcal T(t,0):\ L^2(M_0)\to L^2(M_t)$ be the fundamental solution to the variational equation $\partial_t u=L_{M_t}u$ and $\cT^*(0,t):\  L^2(M_t)\to L^2(M_0)$ its adjoint solving the conjugate equation $\partial_t v=-L_{M_t}v+\tilde H^2 v,$ where $\tilde H=H-\langle x,\mathbf n(x)\rangle/2$ is the rescaled mean curvature. The extra factor takes into account of the fact $\frac{d}{dt} d\mu_{M_t}=-\tilde H^2d\mu_{M_t}$. Denoting $\square_t=\partial_t -L_{M_t}$ and $\square^*_{t}=-\partial_t -L_{M_t}+\tilde H^2, $ we get the Duhamel principle (c.f. Lemma 26.1 of \cite{CC})
	$$\int_0^T dt \int_{M_t}  (\square_t A(x,t) B(x,t)-A(x,t) \square^*_t B(x,t))d\mu_{M_t}=\int_{M_t}A(x,t)B(x,t)d\mu_{M_t} \Big|_{t=0}^{t=T}$$
	for all $A,B:\ \cup_{t\in [0,T]} M_t\to \R$ that are $C^2$ in $x$ and $C^1$ in $t$. 
	
	This implies that in particular $\langle v, \mathcal T(t,0) u\rangle_{L^2(M_t)}=\langle \mathcal T^*(0,t)v,  u\rangle_{L^2(M_0)} $, where $u\in L^2(M_0)$ and $v\in L^2(M_t)$.  We refer readers to \cite[Chapter 26]{CC} for heat kernel on evolving manifolds.
	
	We first give the proof of Lemma \ref{LmHV} following Lemma 5.2 of \cite{HV}.
	\begin{proof}[Proof of Lemma \ref{LmHV}]
		Suppose $v\in L^2(M_t)$ satisfies $\langle v,\cT(t,0)u\rangle=0$ for all $u\in L^2(M_0)$, where $\langle\cdot,\cdot\rangle$ is the $L^2(M_t)$-inner product. Then we get $\langle \cT^*(0,t)v,u\rangle=0$ for all $u\in L^2(M_0)$, which implies $\cT^*(0,t)v=0$. We next denote $w(s)=\cT^*(t-s,t)v$ treating $t$ as fixed. Thus we get $w(t)=0$ and $w(0)=v$ and $w$ solves the equation $\partial_s w=L_{M_{t-s}} w-\tilde H^2_{M_{t-s}} w. $  Applying the following backward uniqueness theorem of Lions-Malgrange, we get $w(0)=v=0$. The proof is then complete.
	\end{proof}
	\begin{theorem}[\cite{LM}]
		\begin{enumerate}
			\item Let $V\subset H$ be two Hilbert spaces such that the injection $V\subset H$ is continuous and $V$ is dense in $H$. 
			\item Let $A(t)$ be a self-adjoint operator such that the function $a(t,u,v):=\langle A(t)u,v\rangle_H$ is a semilinear form continuous on $V\times V$,  is $C^1$ in $t\in [0,T]$, and there exists $C,\lambda, \al>0$ such that $\al \|v\|_V^2\leq a(t,u,u)+\lambda\|u\|_H^2$.
			\item Let $u$ satisfy $u\in L^2((0,T), V)$,$u'\in L^2((0,T),H)$,  $u\in \mathrm{Dom}(A(t))$ for each $t\in (0,T)$, and $\partial_tu+A(t)u=0$, $u(T)=0$.  
		\end{enumerate}
		Then $u\equiv0$ on $[0,T]$.
	\end{theorem}
	In our case, we have $V=H^1(M_t)$, $H=L^2(M_t)$ and $A(s)=L_{M_{t-s}}-\tilde H^2_{M_{t-s}}$.
	
	\begin{proof}[Proof of Lemma \ref{lem:dense}]
		We adapt the above proof as follows. Suppose $v\in H^1(M_t)$ satisfies $\langle v,\cT(t,0)u\rangle_Q=0$ for all $u\in L^2(M_0)$, where $\langle\cdot,\cdot\rangle_Q$ is the $Q(M_t)$-inner product. From the definition of the $Q$-norm,
		$$\langle v_1,v_2\rangle_{Q(M_t)}:=\frac{1}{2}(\|v_1+v_2\|^2_{Q(M_t)}-\|v_1\|^2_{Q(M_t)}-\|v_2\|^2_{Q(M_t)})=\int_{M_t}(\nabla v_1\cdot\nabla v_2+\Lambda v_1v_2)e^{-\frac{|x|^2}{4}}d\mu.$$
		Thus we have
		$\langle (\Lambda-\Delta)v, \cT(t,0)u\rangle=0$ where $\langle\cdot,\cdot\rangle$ is the $L^2(M_t)$-inner product, and $ (\Lambda-\Delta)v\in H^{-1}(M_t)$. Applying the adjoint heat kernel we get $\langle \cT^*(0,t)((\Lambda-\Delta)v), u\rangle=0$. We define similarly $w(s)=\cT^*(s,t)((\Lambda-\Delta)v)$, which satisfies $w(t)=0$ and $w(0)=(\Lambda-\Delta)v$ and $w$ solves the equation $\partial_s w=L_{M_{t-s}} w-\tilde H^2_{M_{t-s}} w. $ The above theorem of Lions-Malgrange implies that $w(0)=(\Lambda-\Delta)v=0$. Then Lax-Milgram implies that $v=0. $ This completes the proof.
	\end{proof}

	\section{Transplantation}\label{AS:Transplantation}
	
	Suppose $\Sigma$ and $\Sigma_1$ are embedded hypersurfaces and $\Sigma_1$ is a graph of function $f$ over $\Sigma$, i.e. $\Sigma_1=\{x+f(x)\bn(x):x\in\Sigma\}$. Then a function $g$ defined on $\Sigma_1$ can be viewed as a function $\bar{g}$ defined on $\Sigma$ as follows: we define $\bar{g}(x)=g(x+f(x)\bn(x))$. Such an identification is called \emph{transplantation} in \cite{SX1}.
	
	The following theorem was proved in the appendix of \cite{SX1}.
	
	\begin{theorem}\label{Thm:Appendix closeness of graphs}
		Let $\Sigma$ be a fixed embedded closed hypersurface. Then given $\eps>0$, there exists a constant $\mu(\eps)>0$ such that the following is true: Suppose $\Sigma_1$ is the graph $\{x+f\bn(x):x\in\Sigma\}$ over $\Sigma$, and $\Sigma_2$ is the graph $\{y+g\bn(y):y\in\Sigma_1\}$ over $\Sigma_1$, and $\|f\|_{C^4(\Sigma)}\leq \mu $, $\|g\|_{C^{2,\alpha}(\Sigma_1)}\leq \mu$, then $\Sigma_2$ is a graph of a function $v$ on $\Sigma$,
		and
		\begin{equation*}
			\|v-(f+g)\|_{C^{2,\alpha}(\Sigma)}\leq \eps\|g\|_{C^{2,\alpha}(\Sigma_1)}.
		\end{equation*}
		Here we transplant $g$ on $\Sigma_1$ to a function on $\Sigma$, and still use $g$ to denote it.
	\end{theorem}
	
	For our purpose, we state the following theorem for the $Q$ estimate. The proof is the same as the proof in \cite{SX1}, because in \cite{SX1} actually a pointwise bound was proved.
	
	\begin{theorem}\label{Thm:Appendix closeness of graphs-Q}
		Suppose the assumptions in Theorem \ref{Thm:Appendix closeness of graphs}. Then we have 
		\begin{equation*}
			\|v-(f+g)\|_Q\leq \eps\|g\|_Q.
		\end{equation*}
		Here we transplant $g$ on $\Sigma_1$ to a function on $\Sigma$ and still use $g$ to denote it.
	\end{theorem}

\end{document}